\documentclass[12pt,a4paper]{amsart}
\usepackage{etex}
\usepackage[utf8]{inputenc}
\usepackage{amssymb,color,esint,epic, graphicx,tikz}

\usepackage{hyperref}
\numberwithin{equation}{section}
\numberwithin{figure}{section}

\theoremstyle{plain}
\newtheorem{thm}{Theorem}
\newtheorem{lem}[thm]{Lemma}
\newtheorem{prop}[thm]{Proposition}

\theoremstyle{definition}

\newtheorem{definition}{Definition}

\theoremstyle{remark}
\newtheorem*{remark}{Remark}

\newcommand{\R}{\mathbb{R}}
\newcommand{\C}{\mathbb{C}}

\newcommand{\B}{\mathbb{B}}
\renewcommand{\P}{\mathcal{P}}

\begin{document}

\title[Interpolation by polynomials in convex domains]{Interpolation by multivariate polynomials in convex domains}

\author{Jorge Antezana}
\address{Departamento de Matem\'atica, Universidad Nacional de La Plata, and Instituto Argentino de Matem\'atica ``Alberto P. Calder\'on'' (IAM-CONICET), Buenos Aires, Argentina}
\email{\href{mailto:antezana@mate.unlp.edu.ar}{\texttt{antezana@mate.unlp.edu.ar}}}

\author{Jordi Marzo}
\address{Dept.\ Matem\`atica i Inform\`atica,
 Universitat  de Barcelona and BGSMath,
Gran Via 585, 08007 Bar\-ce\-lo\-na, Spain}
\email{\href{mailto:jmarzo@ub.edu}{\texttt{jmarzo@ub.edu}}}

\author{Joaquim Ortega-Cerd\`a}
\address{Dept.\ Matem\`atica i Inform\`atica,
 Universitat  de Barcelona and BGSMath,
Gran Via 585, 08007 Bar\-ce\-lo\-na, Spain}
\email{\href{mailto:jortega@ub.edu}{\texttt{jortega@ub.edu}}}

\thanks{Supported by the the Spanish Ministerio de Econom\'ia y Competividad 
(grant MTM2017-83499-P) and the Generalitat de Catalunya (grant 2017 SGR 358)}

\begin{abstract}
Let $\Omega$ be a convex open set in $\mathbb R^n$ and let $\Lambda_k$ 
be a finite subset of $\Omega$. We find necessary geometric conditions for 
$\Lambda_k$ to be interpolating for the space of multivariate polynomials of 
degree at most $k$. Our results are asymptotic in $k$. The density conditions 
obtained match precisely the necessary geometric conditions that sampling sets 
are known to satisfy and they are expressed in terms of the equilibrium 
potential of the convex set. Moreover we prove that in the particular case 
of the unit ball, for $k$ large enough, there is no family of orthogonal reproducing kernels in the 
space of polynomials of degree at most $k$. 
\end{abstract}
\date{\today}
\maketitle

\section{Introduction}

Given a measure $\mu$ in $\mathbb R^n$ we consider the space $\mathcal P_k$ of 
polynomials of total degree at most $k$ in  
$n$-variables endowed with the natural scalar product in $L^2(\mu)$.  We assume 
that $L^2(\mu)$ is a norm for $\mathcal P_k$, i.e. the support of $\mu$ is not 
contained in the zero set of any $p\in \P_k$, $p\ne 0$. In this case the point 
evaluation at any given point $x\in \R^n$ is a bounded linear functional and 
$(\P_k, L^2(\mu))$ becomes a reproducing kernel Hilbert space, i.e for any 
$x\in \R^n$, there is a unique function $K_k(\mu, x,\cdot)\in \mathcal P_k$ such that
\[
 p(x) = \langle p, K_k(\mu,x,\cdot)\rangle = \int p(y) K_k(\mu,x, y) \, d\mu(y).
\]
Given a point $x\in \R^n$ the normalized reproducing kernel is denoted by 
$\kappa_{k,y}$, i.e.
\[
 \kappa_{k,y}(\mu,x) = \frac{K_k(\mu,x,y)}{\|  K_k(\mu,x,\cdot) \|_{L^2(\mu)}} = \frac{K_k(\mu,x,y)}{\sqrt{K_k(\mu,x,x)}}.
\]
We will denote by $\beta_k(\mu,x)$ the value of the reproducing kernel in the diagonal
$$\beta_k(\mu,x)=K_k(\mu,x,x).$$
The function $1/\beta_k(\mu,x)$ is the so called Christoffel function.
For brevity we may omit sometimes the dependence on $\mu$.

\medskip

Following Shapiro and Shields in \cite{shsh} we define sampling and interpolating sets: 
\begin{definition}\label{definterp}
A sequence $\Lambda = \{\Lambda_{k}\}$ of finite sets of points on $\R^n$ is 
said to be
\emph{interpolating} for $(\mathcal P_{k},L^2(\mu))$ if the associated family 
of 
normalized  reproducing
kernels at the points $\lambda\in \Lambda_k$, i.e. $\kappa_{k,\lambda},$ 
is a Riesz sequence in the Hilbert space $\P_k$, uniformly in $k$, i.e there is 
a constant $C>0$ independent of $k$ such that for any linear combination of the 
normalized reproducing kernels we have:
\begin{equation}			\label{ineq:interp}
\frac 1{C} \sum_{\lambda\in\Lambda_k} |c_\lambda|^2 \le \bigl\|
\sum_{\lambda\in \Lambda_k} c_\lambda
\kappa_{k,\lambda} \bigr\|^2\le C   \sum_{\lambda\in\Lambda_k} |c_\lambda|^2,\quad
\forall \{c_\lambda\}_{\lambda\in\Lambda_k}.
 \end{equation}
\end{definition}

%================================================================================================================

The  definition above is 
usually decoupled in two separate conditions. 
The  left hand side  
inequality in (\ref{ineq:interp})
is usually called the \emph{Riesz-Fischer} property for the reproducing kernels 
and it is equivalent to the fact that the following moment problem is solvable: 
for arbitrary values $\{v_\lambda\}_{\lambda_\in \Lambda_k}$ there exists a 
polynomial 
$p\in \P_k$ such that $p(\lambda)/\sqrt{\beta_k(\lambda)} = \langle p, 
\kappa_{k,\lambda}\rangle = v_\lambda$ for all $\lambda\in \Lambda_k$ and 
\[
\|p\|^2 \le C \sum_{\lambda\in \Lambda_k} |v_\lambda|^2 = \sum_{\lambda\in 
\Lambda_k} \frac{|p(\lambda)|^2}{\beta_k(\lambda)}. 
\]
This is the reason $\Lambda$ is 
called an interpolating family.

%===================================================================

The right hand side inequality in (\ref{ineq:interp})
is called the Bessel property for the normalized reproducing kernels 
$\{\kappa_{k,\lambda}\}_{\lambda\in \Lambda_k}$. The Bessel property is equivalent to have
\begin{equation}\label{Carlesonseq}
 \sum_{\lambda\in\Lambda_k} \frac{|p(\lambda)|^2}{\beta_k(\lambda)} \le C  
\|p\|^2
\end{equation}
for all $p\in \mathcal P_k.$ That is, if we denote $\mu_k := \sum_{\lambda\in \Lambda_k} 
\frac{\delta_{\lambda}}{\beta_k(\lambda)}$, we are requiring that the identity 
is a 
continuous embedding of $(\P_k, L^2(\mu))$ into $(\P_k, L^2(\mu_k))$.

%================================================================================================================

\medskip

The notion of sampling play a similar but opposed role.

\begin{definition}         \label{defsampling}
A sequence $\Lambda = \{\Lambda_{k}\}$ of finite sets of points on $\R^n$ is 
said to be
\emph{sampling} or \emph{Marcinkiewicz-Zygmund} for $(\mathcal P_{k},L^2(\mu))$ 
if the associated family of normalized  reproducing
kernels at the points $\lambda\in \Lambda_k$, $\kappa_{k,\lambda}(x)$ 
is a frame in the Hilbert space $\P_k$, uniformly in $k$, i.e there is 
a constant $C>0$ independent of $k$ such that for any polynomial $p\in P_k$:
\begin{equation}        \label{def:sampling}
\frac 1{C} \sum_{\lambda\in\Lambda_k} |\langle p, \kappa_{k,\lambda} \rangle|^2 \le 
\| p\|^2 \le C   \sum_{\lambda\in\Lambda_k} |\langle p, 
\kappa_{k,\lambda} \rangle|^2,\quad
\forall p\in \P_k.
\end{equation}
\end{definition}

Observe that the left hand side inequality in (\ref{def:sampling}) is the Bessel condition mentioned above.  If we were considering a 
single space of polynomials $\mathcal P_{k_0}$ then the notion of interpolating 
family amounts to say that the corresponding reproducing kernels are 
independent. On the other hand, the notion of sampling family corresponds to the 
reproducing kernels span the whole space $\P_{k_0}$. 

In this work we will restrict our attention to two classes of measures:
\begin{itemize}
 \item The 
first is $d\mu(x) = \chi_\Omega(x) dV(x)$ where $\Omega$ is a 
smooth bounded convex domain and $dV$ is the Lebesgue measure.
\item The second is of the form $d\mu(x) = (1-|x|^2)^{a-1/2}\chi_\B(x) dV(x)$ where 
$a\ge 0$ and $\B$ is the unit ball $\B = \{x\in \R^n : |x| \le 1\}$.  
\end{itemize}

In these two cases there are good explicit estimates for the size of the 
reproducing kernel on the diagonal $K_k(\mu,x,x),$ and therefore both notions, 
interpolation and sampling families, become more tangible. 
%Moreover, as will be clear from the proofs, 
%in some cases (when the properties are local) it is enough to work in balls of a fixed radius inside the convex set
In \cite{BerOrt} the authors obtained  necessary geometric conditions for sampling families in bounded smooth convex sets with weights when
the weights satisfy two technical conditions: Bernstein-Markov and moderate growth. These properties 
are both satisfied for the Lebesgue measure in a convex set. 
%In \cite{BerOrt} the authors obtained necessary geometric conditions for 
%sampling and interpolating families when the measure is the Lebesgue measure 
%supported on a compact smooth algebraic variety. They also obtained similar 
%necessary density conditions for sampling families in bounded smooth convex sets with weights when
%the weights
%satisfy two technical conditions: Bernstein-Markov and moderate growth. These properties 
%are both satisfied for the Lebesgue measure in a convex set. 
The case of interpolating families in convex sets was not considered, since there were 
several technical hurdles to apply the same technique.

Our aim in this paper is to fill this gap and obtain necessary geometric 
conditions for interpolating families in the two settings mentioned above.
%We want to find necessary geometric conditions that a sequence of finite sets 
%$\{\Lambda_k\}$ must satisfy if it is interpolating.
The geometric conditions that usually appear in this type of problem come into 
three flavours:
\begin{itemize}
 \item A separation condition. This is implied by the Riesz-Fischer 
condition i.e. the left hand side of (\ref{ineq:interp}). The fact that one should be able to interpolate the 
values one and 
zero implies that different points $\lambda, \lambda' \in \Lambda_k$ with 
$\lambda \ne \lambda'$ cannot be too close. The separation conditions in our settings are studied 
in Section 3.1.

\item A Carleson type condition. This is a condition that ensures the 
continuity of the embedding as in (\ref{Carlesonseq}). A geometric 
characterization of the Carleson is given in  Theorem \ref{Carleson} 
for convex domains and the Lebesgue measure, and in Theorem
\ref{Carleson en bolas} for the ball and the measures $\mu_a$. 

\item A density condition. This is a global condition that usually follows from 
both the Bessel and the Riesz-Fischer condition. A density necessary condition 
for interpolating sequences is provided in Theorem \ref{density} for convex sets 
endowed with the Lebesgue measure, and in Theorem \ref{teo:density_ball} for 
the ball and the measures $\mu_a$. Moreover, in this last setting we get an extension 
of the density results proved in \cite{BerOrt} for sampling sequences.
\end{itemize}

Finally, a natural question is whether or not there exists a family $\{\Lambda_k\}$ that is both sampling and interpolating. To answer this question is very difficult  in general \cite{OUbook}. A particular case 
is when $\{\kappa_{k,\lambda}\}_{\lambda\in \Lambda_k}$ form an orthonormal basis. 
In the last section we study the existence of orthonormal basis of reproducing kernels in the case of the ball 
with the measures $\mu_a$. More precisely, if the spaces $\P_k$ endowed with the inner product of $L^2(\mu_a)$, then in Theorem \ref{noONbasis} we prove that for $k$ big enough the space $\P_k$ does 
not admit an orthonormal basis of reproducing kernels. To
determine whether or not there exists a family $\{\Lambda_k\}$ that is both sampling and interpolating for $(\P_k,\mu_a)$ remains an open problem.

\section{Technical results}

Before stating and proving our results we will recall the behaviour of the kernel in the diagonal, or equivalently the Christoffel function, we will define an appropriate metric and introduce some needed tools.

\subsection{Christoffel functions and equilibrium measures}

To write explicitly the sampling and interpolating conditions we need an estimate of the Christoffel function. In \cite{BerOrt} it was observed that in the case of the measure
$d\mu(x) = \chi_\Omega(x) dV(x)$ it is possible to obtain precise estimates for 
the size of the reproducing kernel on the diagonal:
\begin{thm}\label{thm:diagonalestimate}
Let $\Omega$ be a smoothly bounded convex domain in $\R^n$. Then the 
reproducing kernel for $(\mathcal P_{k},\chi_\Omega dV)$ satisfies
\begin{equation}\label{kernelconvex}
\beta_k(x)=K_k(x,x) \simeq \min\Bigl( 
\frac{k^n}{\sqrt{d(x,\partial \Omega)}}, k^{n+1}\Bigr)\quad
\forall x\in\Omega.
\end{equation}
where  $d(x,\partial \Omega)$ denotes the Euclidean distance of $x\in \Omega$ to the boundary 
of $\Omega$.
% Moreover, a necessary condition for the sequence $\Lambda_{k}$ of sets
% of points on $\Omega$ to be sampling for $H_k(M)$ is that the density of 
% sampling points is at least equal to the density of the equilibrium measure
% $\mu_{eq}$ of $\Omega$, as $k\rightarrow\infty$, i.e. the pluripotential 
% Nyquist bound \eqref{Nyquist} holds.
\end{thm}

For the weight $(1-|x|^2)^{a-1/2}$ in the ball $\mathbb B$ the asymptotic behaviour of the Christoffel is well known.

%We are not going to use them but there are even explicit expressions for the reproducing kernels in terms 
%of integrals of Jacobi polynomials, \cite[Theorem 3.3.]{Xu99}. T

% Following \cite{petxu} define for $x\in B^d$
% $$\mathcal{W}_\mu (n,x)=\left( \sqrt{1-|x|^2}+n^{-1}  \right)^{2\mu}\approx \max\left\{ d(x,\partial B^d )^{\mu} ,n^{-2\mu} \right\}.$$

%===================================================================

\begin{prop}					\label{prop:Christoffel}
  For any $a \ge  0$ and $d\ge 1$ let $$d\mu_a(x)=(1-|x|^2)^{a-1/2}\chi_\B(x) dV(x).$$ Then the reproducing kernel for $(\mathcal P_{k},d\mu_a )$ satisfies
  \begin{equation}\label{kernelball}
\beta_k(\mu_a,x)=K_k(\mu_a,x,x) \simeq \min\Bigl( 
\frac{k^n}{d(x,\partial \mathbb B)^a}, k^{n+2a}\Bigr)\quad
\forall x\in\Omega.
\end{equation}

\end{prop}

The proof follows from \cite[Prop 4.5 and 5.6]{petxu}, Cauchy–Schwarz inequality and the extremal characterization of the kernel
$$K_k(\mu_a;x,x)=\left\{ |P(x)|^2\;\; : \;\;P\in \mathcal P_k,\int |P|^2 d\mu_a \le 1  \right\}.$$

To define the equilibrium measure we have to introduce a few concepts from pluripotential theory, see \cite{Kli91}. Given a non pluripolar compact set $K\subset \R^n\subset \C^n$ the pluricomplex Green function is the semicontinuous regularization 
$$G^*_K(z)=\limsup_{\xi \to z}G_K(\xi),$$
where
$$G_K(\xi)=\sup \left\{ \frac{\log^+|p(\xi)|}{\deg(p)}\; : \; p\in P(\C^n),\; \sup_K |p(\xi)|\le 1\right\}.$$
The pluripotential equilibrium measure for of $K$ is the (probability) Monge-Amp\`ere Borel measure 
$$d \mu_{eq}=(d d^c G_K^*)^n.$$

In the general case, when $\Omega$ is a smooth bounded convex domain the equilibrium 
measure is very well understood, see \cite{BedTay} and \cite{blmr}. It behaves 
roughly as $d\mu_{eq}\simeq 1/\sqrt{d(x,\partial\Omega)}dV$. In particular,
the pluripotential equilibrium measure for the ball $\mathbb B$ is given (up to normalization) by $d \mu_0(x)=\frac{1}{\sqrt{1-|x|^2}}dV(x).$  

% The behaviour of the Christoffel function in the equilibrium case is not known but one can expect that 
% for $\Omega$ a smoothly bounded convex domain in $\R^n$ the
% reproducing kernel for $(\mathcal P_{k},d\mu_{eq})$
% $$K_k(\mu_{eq},x,x) = \|K_k(\mu_{eq} ,x,\cdot)\|_{L^2(\mu_{eq})}^2 \simeq k^n,$$
% It is known that this is also the behaviour for any measure similar to the Lebesgue measure inside the support, \cite[Lemma 5.2]{KL13}.

\subsection{An anisotropic distance}\label{metric}
The natural distance to formulate the separation condition and the 
Carleson condition is not the Euclidean distance. Consider in the unit ball 
$\B\subset \R^n$ the following distance: 
\[
 \rho(x,y) =  \arccos\left\{\langle x, y\rangle + \sqrt{1-|x|^2}\sqrt{1-|y|^2}\right\}.
\]
This is the geodesic distance of the points $x',\ y'$ in the 
sphere $\mathbb S^{n}$ defined as $x'=(x,\sqrt{1-|x|^2})$ and 
$y'=(y,\sqrt{1-|y|^2})$.
If we consider anisotropic balls $B(x,\varepsilon) = \{y \in \mathbb B: 
\rho(x,y) < \varepsilon\}$, 
 they are comparable to a box centered at $x$ (a 
product of intervals) which are of size $\varepsilon$ in the tangent directions 
and 
size $\varepsilon^2 + 
\varepsilon\sqrt{1-|x|^2}$ in the normal direction. If we want to refer to a 
Euclidean 
ball of center $x$ and radius $\varepsilon$ we would use the notation 
$\B(x,\varepsilon)$.

The Euclidean volume of a ball $B(x,\varepsilon)$ is comparable to 
$\varepsilon^{n}\sqrt{1-|x|^2}$ if  $(1-|x|^2) > \varepsilon^2$ and 
$\varepsilon^{n+1}$ otherwise. 

This distance $\rho$ can be extended to an arbitrary smooth convex domain 
$\Omega$ by using Euclidean balls contained in $\Omega$ and tangent to the 
boundary of $\Omega$. This can be done in the following way. Since $\Omega$ is 
smooth, there is a tubular neighbourhood $U\subset \R^n$ of the boundary of 
$\Omega$ where each point $x\in U$ has a unique closest point $\tilde x$ in 
$\partial \Omega$ and the normal line to $\partial \Omega$ at $\tilde x$ passes 
by $x$. There is a fixed small radius $r>0$ such that for any point $x\in U\cap 
\Omega$ it is contained in a ball of radius $r$, $B(p, r)\subset \Omega$ and 
such that it is tangent to $\partial \Omega$ at $\tilde x$. We define on $x$ a 
Riemannian metric which comes from the pullback of the standard metric on 
$\partial \tilde B(p,r)$ where $\tilde B(p,r)$ is a ball in $\R^{n+1}$ centered 
at $(p,0)$ and of radius $r>0$ by the projection of $\R^{n+1}$ onto the first 
$n$-variables. In this way we have defined a Riemannian metric in the domain 
$\Omega\cap U$. In the core of $\Omega$, i.e. far from the boundary we use the 
standard Euclidean metric. We glue the two metrics with a partition of unity. 

The resulting metric $\rho$ on $\Omega$ has the relevant property that the 
balls of radius $\epsilon$ behave as in the unit ball, that is a ball 
$B(x,\varepsilon)$ of center $x$ and of radius $\varepsilon$ in this metric is 
comparable to a box of size 
$\varepsilon$ in the tangent directions 
and size $\varepsilon^2 + \varepsilon\sqrt{d(x,\partial \Omega)}$ in the normal 
direction.% where $d(x)$ is the Euclidean distance from $x$ to the boundary of $\Omega$.

\subsection{Well localized polynomials}
The basic tool that we will use to prove the Carleson condition and the 
separation are well localized polynomials. These were studied by Petrushev and 
Xu in the unit ball with the measure $d\mu_a=(1-|x|^2)^{a-\frac{1}{2}}dV,$ for $a\ge 0.$ 
We recall their basic properties:

\begin{thm}[Petrushev and Xu]	\label{theorem_PX}
Let $d\mu_a=(1-|x|^2)^{a-\frac{1}{2}}dV$ for $a\ge 0.$
For any $k\ge 1$ entire and any $y\in \B\subset \R^n$ there are polynomials $L_k^a (\cdot , y) \in \mathcal P_k$ that 
satisfy:
\begin{enumerate}
 \item $L_k^a$ as a variable of $x$ is a polynomial of degree $2k$.
 \item $L_k^a(x,y) = L_k^a(y, x)$.
 \item $L_k^a$ reproduces all the polynomials of degree $k$, i.e.
\begin{equation}							\label{reproducing}
  p(y) = b_n^a \int_\B L_k^a(x,y) p(x)\, d\mu_a(x).\qquad \forall p\in \mathcal P_k.
\end{equation}

\item For any $\gamma>0$ there is a $c_\gamma$ such that 
\begin{equation}		\label{offdiag}
 |L_k^a (x, y)| \le c_\gamma \frac{\sqrt{\beta_k(\mu_a,x) \beta_k(\mu_a,y) }}{(1+ k 
\rho(x,y))^\gamma}.
\end{equation}
\item The kernels $L_k^a$ are Lispchitz with respect to the metric $\rho$, more 
concretely, for all $x\in B(y, 1/k)$:
\begin{equation}\label{integ_weight}
 |L_k^a (w, x) - L_k^a (w, y)| \le c_\gamma \frac{k
\rho(x,y) \sqrt{ \beta_k(\mu_a,w) \beta_k(\mu_a,y) }}{ (1+k \rho(w,y))^\gamma}
\end{equation}
\item						\label{diagonal_weight} 
There is $\varepsilon > 0$  such that $L_k^a (x, y) \simeq 
K_k(\mu_a; y, y)$ for all
$x\in B(y, \varepsilon/k)$.
\end{enumerate}
\end{thm}

%==============================================================================

\begin{proof}
 All the properties are proved in \cite[Thm 4.2, Prop 4.7 and 4.8]{petxu} except the behaviour near the diagonal
number~\ref{diagonal_weight}. Let us start by 
observing that by the Lipschitz condition \eqref{integ_weight} it is enough to prove 
that $L_k^a (x,x)\simeq K_k(\mu_a;x,x)$.  

This follows from the definition of $L_k^a$  which is done as follows.
The subspace $V_k\subset L^2(\B)$ are the polynomials of degree $k$ that are 
orthogonal to lower degree polynomials in $L^2(\B)$ with respect to the measure $d\mu_a$. 
Consider the kernels $P_k(x, y)$ 
which are the kernels that give the orthogonal projection on $V_k$. If 
$f_1,\ldots, f_r$ is an orthonormal basis for $V_k$ then $P_k(x,y) = \sum_{j= 
1}^r f_j(x)f_j(y)$.  The kernel $L_k^a$ is defined as
\[
 L_k^a (x,y) = \sum_{j = 0}^\infty \hat a \left(\frac j k\right) P_j(x,y).
\]
We assume that $\hat a$ is compactly supported, $\hat a \ge 0$, $\hat a \in 
\mathcal C^\infty(\mathbb R)$, $\operatorname{supp} \hat a \subset [0, 2]$, 
$\hat a(t) = 1$
on $[0,1]$ and $\hat a (t) \le 1$ on $[1,2]$ as in the picture:

\begin{center}
\ifx\XFigwidth\undefined\dimen1=0pt\else\dimen1\XFigwidth\fi
\divide\dimen1 by 3174
\ifx\XFigheight\undefined\dimen3=0pt\else\dimen3\XFigheight\fi
\divide\dimen3 by 2499
\ifdim\dimen1=0pt\ifdim\dimen3=0pt\dimen1=4143sp\dimen3\dimen1
  \else\dimen1\dimen3\fi\else\ifdim\dimen3=0pt\dimen3\dimen1\fi\fi
\tikzpicture[x=+\dimen1, y=+\dimen3]
{\ifx\XFigu\undefined\catcode`\@11
\def\temp{\alloc@1\dimen\dimendef\insc@unt}\temp\XFigu\catcode`\@12\fi}
\XFigu4143sp
% Uncomment to scale line thicknesses with the same
% factor as width of the drawing.
%\pgfextractx\XFigu{\pgfqpointxy{1}{1}}
\ifdim\XFigu<0pt\XFigu-\XFigu\fi
\clip(2913,-4962) rectangle (6087,-3100);
\tikzset{inner sep=+0pt, outer sep=+0pt}
\pgfsetlinewidth{+7.5\XFigu}
\draw (2925,-4725)--(6075,-4725);
\draw (3150,-4950)--(3150,-3200);
\draw (4500,-4770)--(4500,-4680);
\draw (5850,-4770)--(5850,-4680);
\draw (3150,-3600)--(4635,-3600);
\draw 
(4635,-3600)--(4636,-3600)--(4640,-3600)--(4649,-3601)--(4664,-3601)--(4684,
-3603)
  
--(4707,-3604)--(4730,-3606)--(4753,-3609)--(4773,-3612)--(4790,-3616)--(4806,
-3620)
  
--(4820,-3625)--(4833,-3631)--(4845,-3638)--(4855,-3644)--(4865,-3652)--(4875,
-3661)
  
--(4886,-3671)--(4896,-3683)--(4907,-3696)--(4918,-3710)--(4929,-3726)--(4940,
-3743)
  
--(4951,-3761)--(4962,-3780)--(4973,-3799)--(4984,-3820)--(4995,-3841)--(5006,
-3862)
  
--(5018,-3885)--(5026,-3902)--(5035,-3920)--(5044,-3939)--(5054,-3959)--(5064,
-3980)
  
--(5074,-4001)--(5085,-4024)--(5095,-4047)--(5107,-4071)--(5118,-4095)--(5129,
-4120)
  
--(5140,-4144)--(5151,-4169)--(5162,-4193)--(5173,-4217)--(5183,-4240)--(5193,
-4263)
  
--(5202,-4285)--(5211,-4306)--(5219,-4326)--(5227,-4346)--(5235,-4365)--(5244,
-4388)
  
--(5253,-4410)--(5261,-4431)--(5269,-4453)--(5278,-4474)--(5286,-4495)--(5294,
-4515)
  
--(5303,-4535)--(5311,-4554)--(5320,-4572)--(5328,-4589)--(5337,-4605)--(5346,
-4619)
  
--(5355,-4632)--(5364,-4644)--(5373,-4655)--(5382,-4664)--(5393,-4673)--(5405,
-4681)
  
--(5418,-4688)--(5432,-4695)--(5449,-4700)--(5468,-4705)--(5489,-4709)--(5513,
-4713)
  
--(5539,-4716)--(5567,-4719)--(5594,-4721)--(5620,-4722)--(5641,-4724)--(5656,
-4724)
  --(5666,-4725)--(5669,-4725)--(5670,-4725);
\pgftext[base,left,at=\pgfqpointxy{4500}{-4590}] 
{$k$};
\pgftext[base,left,at=\pgfqpointxy{5850}{-4545}] 
{$2k$};
\pgftext[base,left,at=\pgfqpointxy{2970}{-3645}] 
{$1$};
\pgftext[base,left,at=\pgfqpointxy{3690}{-3465}] 
{$\hat a(x/k)$};
\endtikzpicture%
\end{center}

\noindent Then, all the terms are positive in the diagonal. Hence, we get
\[
 \beta_k (\mu_a,x) = K_k(\mu_a;x, x) \le L_k^a (x, x) \le K_{2k}(\mu_a;x,x) = \beta_{2k}(\mu_a,x).
\]
Since $\beta_k (\mu_a,x) \simeq \beta_{2k} (\mu_a,x)$ we obtain the desired estimate.

\end{proof}

%==============================================================================

They also proved the following integral estimate \cite[Lemma 4.6]{petxu}

\begin{lem}					\label{lem3_weight}
Let $\alpha>0$ and $a\ge 0.$ If $\gamma>0$ is big enough we have
 \[
  \int_{\B} \frac{K_k(\mu_a,y,y)^\alpha }{(1+k \rho(x,y))^\gamma}d\mu_a (y) \lesssim 
\frac 1 {K_k(\mu_a,x,x)^{1-\alpha}}.
 \]
\end{lem}

\section{main results}

\subsection{Separation}

In our first result we prove that for $\Lambda=\{ \Lambda_k \}$ interpolating there exist $\epsilon>0$ such that
$$\inf_{\lambda,\lambda'\in \Lambda_k,\lambda\neq \lambda'}\rho(\lambda,\lambda')\ge \frac{\epsilon}{k}.$$

%We consider first the weighted case in the ball.

%==============================================================================
% 
% 
% \begin{thm}
%  Let $\Lambda$ be an interpolating sequence for  $\mathcal P_k$ with weight $d\mu_a(x)=(1-|x|^2)^{a-\frac{1}{2}}dV(x)$ then there is an $\varepsilon>0$ such that the balls 
% $\{B(\lambda,\varepsilon/k)\}_{\lambda\in \Lambda_k}$ are pairwise disjoint.
% \end{thm}
% 
% %==============================================================================
% 
% 
% \begin{proof}
% Assume that there exist $\lambda', \lambda\in \Lambda_k$ with $\lambda' \in B(\lambda,\varepsilon/k)\setminus \{\lambda\}$. Since $\Lambda_k$ is 
% interpolating we can build a polynomial 
% $p\in \Pi_k^d$ such that $p(\lambda') = 0$, $p(\lambda) = 1$ and
% $\|p\|^2_\mu \lesssim 1/K_k(\mu ;\lambda, \lambda)$. By the reproducing property (\ref{reproducing})
% \[
% 1 =b_d^\mu \int_{B^d} (L_k^\mu(\lambda, w)-L_k^\mu(\lambda', w)) p(w) d\Omega_\mu (w).
% \]
% We can use again the estimate $|p(w)| \le \sqrt{B_k^\mu (w)} \|p\|_\mu \le
% \sqrt{B_k^\mu (w)/B_k^\mu (\lambda)}$ and the inequality \eqref{integ_weight}
% to obtain 
% \[
%  1 \lesssim k\rho(\lambda,\lambda') \int_{B^d} \frac{B_k(w) 
% d\Omega_\mu (w)}{(1+k\rho(y,\lambda))^\beta}
% \]
% Taking $\alpha = 1$ in Proposition~\ref{integ_weight} we obtain $1 \lesssim 
% k\rho(\lambda,\lambda')$
% as stated.
% \end{proof}
% 

%==============================================================================

%Let us prove the separation property.
\begin{thm}\label{separadas}
 If $\Omega$ is a smooth convex set and $\Lambda=\{ \Lambda_k \}$ is an interpolating 
sequence then there is an $\varepsilon>0$ such that the balls 
$\{B(\lambda,\varepsilon/k)\}_{\lambda\in \Lambda_k}$ are pairwise disjoint.
\end{thm}

\begin{proof}
Consider the metric in $\Omega$ defined in section \ref{metric}. We can restrict the argument to a ball, of a fixed radius $r(\Omega),$ in one of the two cases:
tangent to the boundary or at a positive distance to the complement $\R^n\setminus \Omega.$ 
Let us assume that there is another point from $\Lambda_k$, $\lambda' \in 
B(\lambda,\varepsilon/k)$. Since it is interpolating we can build a polynomial 
$p\in \mathcal P_k$ such that $p(\lambda') = 0$, $p(\lambda) = 1$ and
$\|p\|^2 \lesssim 1/K_k(\mu_{\frac{1}{2}},\lambda, \lambda)$. Take a ball $\Omega$ such 
that it contains $\lambda$ and $\lambda'$ and that it is tangent to $\partial 
\Omega$ at a closest point to $\lambda$.
To simplify the notation assume that radius of this ball is one, and it is denoted by $\B.$
In this ball the kernel $L_k^{\frac{1}{2}}$ from Theorem \ref{theorem_PX}, for the Lebesgue measure $a=\frac{1}{2},$ is 
reproducing so
\begin{equation} 		\label{eq1:separation}
1 = \int_\B (L_k^{\frac{1}{2}}(\lambda, w)-L_k^{\frac{1}{2}}(\lambda', w)) p(w) dV(w). 
\end{equation}
We can use the estimate $$|p(w)| \le \sqrt{\beta_k(\mu_{\frac{1}{2}},w)} \|p\| \le
\sqrt{\beta_k(\mu_{\frac{1}{2}},w)/\beta_k(\mu_{\frac{1}{2}},\lambda))}$$ and the inequality (\ref{integ_weight})
to obtain 
\[
 1 \lesssim k\rho(\lambda,\lambda') \int_\B \frac{\beta_k(\mu_{\frac{1}{2}},w) 
dV(w)}{(1+k\rho(y,\lambda))^\gamma},
\]
Taking $\alpha = 1$ and $a=\frac{1}{2}$ in Lemma~\ref{lem3_weight} we obtain $1 \lesssim 
k\rho(\lambda,\lambda')$
as stated.
\end{proof}

%==============================================================================

Observe that considering the general case $L_k^a$ in (\ref{eq1:separation}), one can prove the corresponding result
for 
interpolating sequences for  $\mathcal P_k$ with weight $d\mu_a(x)=(1-|x|^2)^{a-\frac{1}{2}}dV(x)$ in the ball $\mathbb B.$

\subsection{Carleson condition}

Let us deal with condition \eqref{Carlesonseq}. For a convex smooth set $\Omega \subset \R^n$
is a particular instance 
of the 
following definition.
\begin{definition}
A sequence of measures $\mu_k \in \mathcal M(\Omega)$ are called Carleson 
measures for $(\mathcal P_k,d\mu )$ if there is a constant $C>0$ such that
\[
 \int_{\Omega} |p(x)|^2\, d\mu_k(x) \le C \|p\|_{L^2(\mu)}^2,
\]
for all $p\in \mathcal P_k$.
\end{definition}
In particular if $\Lambda_k$ is a sequence of interpolating sets then the 
sequence of measures
$\mu_k = \sum_{\lambda\in \Lambda_k} \frac{\delta_\lambda}{\beta_k(\lambda)}$
is Carleson.

The geometric characterization of the Carleson measures when $\Omega$ is a 
smooth convex domain is in terms of anisotropic balls.

\begin{thm}\label{Carleson} A sequence of measures $\mu_k$ is Carleson for the polynomials 
$\P_k$ in a smooth bounded convex domain $\Omega$ if and only if there is a constant 
$C$ such that for all points $x\in \Omega$ 
\begin{equation}\label{geonec} 
\mu_k(B(x,1/k)) \le C V(B(x,1/k)).
\end{equation}

\end{thm}

%============================================================================

\begin{proof}
We prove the necessity. For any $x\in\Omega$ there is a cube $Q$ that contains 
$\Omega$ which is tangent to $\partial\Omega$ at a closest point to $x$ as in 
the picture:
\begin{figure}
\includegraphics[width=4cm]{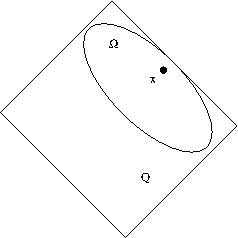}
\caption{}
\end{figure}
This cube has fixed dimensions independent of the point $x\in\Omega$. We can 
construct a polynomial $Q_k^x$ of degree at most $k n$ taking the product of one 
dimensional polynomials $L_k^{\frac{1}{2}}$. We test against these 
polynomials that peak at $B(x,1/k)$ 
$$\int_{B(x,1/k)} |Q_k^x|^2 d\mu_k \le \int_{\Omega} |Q_k^x|^2 d\mu_k\le C \| Q_k^x \|_{L^{2}(Q)},$$
by property $(6)$ in Theorem \ref{theorem_PX} and the estimate (\ref{kernelball}) the necessary condition follows.

For the sufficiency we use the reproducing property of  $L_k^{\frac{1}{2}}(z,y)$. That is for 
any point $x\in \Omega$ there is a Euclidean ball $\B_x$ contained in $\Omega$ 
such that $x\in \B_x$ and it is tangent to $\partial \Omega$ in the closest 
point to $x$ as in the picture. Moreover since $\Omega$ is a smoothly bounded 
convex domain we can assume that the radius $\B$ has a lower bound independent 
of $x$. In this ball we can reconstruct any polynomial $p\in \mathcal{P}_k$ using $L_k^{\frac{1}{2}}$. That 
is 
\[
 \int_\Omega |p(x)|^2 \, d\mu_k(x) \le
 \int_\Omega \left|\int_{B_x} L_{2k}^{\frac{1}{2}}(x,y) p^2(y)\, dV(y)\right|\, d\mu_k(x).
\]
We use the estimate (\ref{offdiag}) and we get
\[
  \int_\Omega |p(x)|^2 \, d\mu_k(x) \lesssim \int_\Omega\int_{B_x} 
\frac{\sqrt{\beta_k(x)\beta_k(y)}}{(1+ k 
\rho(x,y))^\gamma} |p(y)|^2 dV(y)\, d\mu_k(x).
\]
We break the integral in two regions, when $\rho(x,y) < 1 $ and otherwise. When 
$k$ is big enough we obtain:
\[\begin{split}
  \int_\Omega |p(x)|^2 \, d\mu_k(x) \le 
  &\int_\Omega\int_{B_x \cap\rho(x,y) > 1}
  |p(y)|^2 dV(y) d\mu_k(x) + \\
  &C\int_\Omega\int_{B_x\cap\rho(x,y) < 
1} 
\frac{\sqrt{\beta_k(x)\beta_k(y)}}{(1+ k 
\rho(x,y))^\gamma} |p(y)|^2 dV(y)\, d\mu_k(x)
\end{split}
\]
The first integral in the right hand side is bounded by $\int_\Omega 
|p(y)|^2\, dV(y)$ since $\mu_k(\Omega)$ is bounded by hypothesis (it is 
possible to cover $\Omega$ by balls $\{B(x_n,1/k)\}$ with controlled overlap). 

In the second integral, observe that if $w\in B(x,1/k)$ then $\rho(w,x)\le 1/k$ 
and therefore
\[
\frac{\sqrt{\beta_k(x)\beta_k(y)}}{(1+ k 
\rho(x,y))^\gamma} \lesssim
\frac 1{V(B(x,1/k))} \int_{B(x,1/k)} \frac{\sqrt{\beta_k(w)\beta_k(y)}}{(1+ k 
\rho(w,y))^\gamma} dV(w).  
\]
We plug this inequality in the second integral and we can bound it by
\[
 C\int_\Omega |p(y)|^2 \int_{\rho(w,y) < 2}\frac{\sqrt{\beta_k(w)\beta_k(y)}}{(1+ k 
\rho(w,y))^\gamma} \frac {\mu_k(B(w,1/k))}{V(B(w,1/k))}dV(w) dV(y).
\]
We use the hypothesis \eqref{geonec} and Lemma~\ref{lem3_weight} with $\alpha = 
1/2$ to bound it finally by $C\int_\Omega |p(y)|^2dV(y)$.

\end{proof}

% When $\Omega$ is a smooth bounded convex domain the equilibrium 
% measure is very well understood, see \cite{BedTay} and \cite{blmr}. It behaves 
% roughly as $d\mu_{eq}\simeq 1/\sqrt{d(x,\partial\Omega)}dV$.

The weighted case in the unit ball is simpler.

%==============================================================================

\begin{thm}\label{Carleson en bolas} 
Let $d\mu_a (x)=(1-|x|^2)^{a-\frac{1}{2}}dV(x)$ for $a\ge 0$ the weight in the unit ball $\B\subset \R^n.$
A sequence of measures $\{ \mu_k \}$ are Carleson for $(\P_k,\mu_a)$ if there is a constant 
$C$ such that for all points $x\in \B$ 
\begin{equation} \mu_k(B(x,1/k)) \le C\; \mu_a(B(x,1/k)).
\end{equation}
\end{thm}

%==============================================================================

\begin{proof}
Supose $\{ \mu_k \}$ are Carleson. Then for any $x\in \B$
\begin{equation}
 \int_{B(x,1/k)} |L_k^a(x,w)|^2\, d\mu_k(w) \le C \|L_k^a(x,\cdot) \|^2_\mu.
\end{equation}
By property (\ref{diagonal_weight}) in Theorem \ref{theorem_PX} and the estimate
$$K_k(\mu,x,x)\le \|L_k^a(x,\cdot) \|^2_\mu\le K_{2k}(\mu,x,x) ,$$
the result follows. The necessity follows exactly like in the unweighted case with the obvious changes.
\end{proof}

%================================================================================

\subsection{Density condition}

In \cite[Theorem 4]{BerOrt} a necessary density condition for 
sampling sequences for polynomials in convex domains was obtained. It states 
the following:

\begin{thm} \label{densitysamp}
Let $\Omega$ be a smooth convex domain in $\R^n$, and let $\Lambda$ be a 
sampling sequence. Then for any $\B(x,r) \subset \Omega$ the following 
holds:
\[
 \limsup_{k\to\infty} \frac{\#\Lambda_k \cap \B(x,r)}{\dim \mathcal P_k}\ge 
\mu_{eq}(\B(x,r)).
\]
Here $\mu_{eq}$ is the equilibrium measure associated to $\Omega$.
\end{thm}

Let us see 
how, with a similar technique, a corresponding density condition can be 
obtained as well in the case of interpolating sequences.

\begin{thm} \label{density}
Let $\Omega$ be a smooth convex domain in $\R^n$, and let $\Lambda$ be an 
interpolating sequence. Then for any $\B(x,r) \subset \Omega$ the following 
holds:
\[
 \limsup_{k\to\infty} \frac{\#\Lambda_k \cap \B(x,r)}{\dim \mathcal P_k}\le 
\mu_{eq}(\B(x,r)).
\]
Here $\mu_{eq}$ is the equilibrium measure associated to $\Omega$.
\end{thm}
\begin{remark}
 In the statements of Theorems~\ref{densitysamp} and \ref{density} we could 
have replaced $\mathbb B(x, r)$ by any open set, in particular they could have been formulated with balls 
$B(x,r)$ in the anisotropic metric.
\end{remark}

%==========================================================================

\begin{proof}
Let $F_k\subset \P_k$ be the subspace spanned by 
\[
 \kappa_\lambda(x)=K_k(\lambda,x)/\sqrt{\beta_k(\lambda)}\qquad \forall
\lambda \in \Lambda_k.
\] 
Denote by $g_\lambda$ the dual (biorthogonal) basis to $\kappa_\lambda$ in
$F_k$. We 
have clearly that 
\begin{itemize}
\item We can span any function in $F_k$ in terms of $\kappa_\lambda$, thus:
\[
 \sum_{\lambda\in \Lambda_k} \kappa_\lambda(x)g_\lambda(x)=\mathcal K_k(x,x),
\]
where $\mathcal K_k(x,y)$ is the reproducing 
kernel of the subspace $F_k$.  
\item The norm of $g_\lambda$ is uniformly bounded since $\kappa_\lambda$ was a 
uniform Riesz sequence.
\item $g_\lambda(\lambda)=\sqrt{\beta_k(\lambda)}$. This is due to the 
biorthogonality and the reproducing property.
\end{itemize}
We are going to prove that the measure $\sigma_k=\frac 1{\dim \mathcal
P_k}\sum_{\lambda\in 
\Lambda_k}\delta_\lambda$, 
and the measure $\nu_k=\frac 1{\dim \mathcal P_k} \mathcal K_k(x,x)d\mu(x)$ 
are very 
close to each other. This are two positive measures that are not probability 
measures but they have the same mass (equal to $\frac{\#\Lambda_k}{\dim 
\mathcal P_k}\le 1$). Therefore, there is a way to quantify the closeness  
through the Vaserstein $1$-distance. For an introduction to Vaserstein distance 
see for instance \cite{Villani}. We want to prove that
$W(\sigma_k,\nu_k)\to 0$ because the Vaserstein distance metrizes the 
weak-* topology.

In this case, it is known that $\mathcal K_k(x,x)\le K_k(x,x)$ and 
$\frac 1{\dim \mathcal P_k} 
\beta_k(x)\, 
\to \mu_{eq}$ in the weak-* topology, where $\mu_{eq}$ is the normalized 
equilibrium measure associated to $\Omega$ (see for instance \cite{BBN11}).  Therefore,
$\limsup_k \sigma_k\le \mu_{eq}$.

In order to prove that $W(\sigma_k,\nu_k)\to 0$ we use a non positive 
transport plan as in \cite{LevOrt}:
\[
 \rho_k(x,y)=\frac 1{\dim \mathcal P_k} \sum_{\lambda\in \Lambda_k} 
\delta_\lambda(y) 
\times g_\lambda(x)\kappa_\lambda(x)\,d\mu(x)
\]
It has the right marginals, $\sigma_k$ and $\nu_k$ 
and we can estimate the integral 
\[
W(\sigma_k,\nu_k)\le
\iint_{\Omega \times \Omega} |x-y|d|\rho_k|=O(1/\sqrt{k}).
\]
The only point that merits a clarification is that we need an inequality:
\[
\begin{split}  \frac 1{\dim \mathcal P_k} \sum_{\lambda\in \Lambda_k} 
\int_\Omega
|\lambda-x|^2 
\frac{|K_k(\lambda,x)|^2}{ K_k(x,x)}\,d\mu(x)\le \\ 
\frac 1{\dim \mathcal P_k} 
\iint_{\Omega \times 
\Omega} |y-x|^2 
|K_k(y,x)|^2\, d\mu(x)d\mu(y).
\end{split}
\]
This is problematic. We know that $\Lambda_k$ is an interpolating 
sequence for the polynomials of degree $k$. Thus the normalized reproducing 
kernels at $\lambda\in\Lambda_k$ form a Bessel sequence for $\P_k$ but the 
inequality that we need is applied to $K_k(x,y)(y_i-x_i)$ for all $i=1,\ldots, 
n$. That is to a 
polynomial of degree $k+1$. We are going to show that if $\Lambda_k$ is an 
interpolating sequence for the polynomials of degree $k$ it is also a Carleson 
sequence for the polynomials of degree $k+1$. 

Observe that since it is interpolating then it is uniformly separated, i.e. 
$B(\lambda, \varepsilon/k)$ are disjoint. That means that in particular 
$$\mu_{k} (B(z, 1/(k+1)) \lesssim V(B(z,1/(k+1)).$$ Thus $\mu_k$  is a 
Carleson measure for $\P_{k+1}$.

Finally in \cite[Theorem~17]{BerOrt} it was proved that 
\[
\frac 1{\dim \mathcal P_k} 
\iint_{\Omega \times 
\Omega} |y-x|^2 
|K_k(y,x)|^2\, d\mu(x)d\mu(y) = O(1/k).
\]
\end{proof}

From the behaviour on the diagonal of the kernel \eqref{kernelball}
its easy to check that the kernel is both Bernstein-Markov (sub-exponential) and has moderate growth, see definitions in \cite{BerOrt}. From the characterization for sampling sequences
proved in \cite[Theorem 1]{BerOrt} and with the obvious changes in the proof of the previous theorem we deduce the following:

%===============================================================================

\begin{thm}						\label{teo:density_ball}
Consider the space of polynomials $\mathcal P_k$ restricted to the ball $\B\subset \R^n$ with the measure $d\mu_a(x)=(1-|x|^2)^{a-\frac{1}{2}}dV.$
Let $\Lambda=\{\Lambda_k \}$ be a sequence sets of points in $\B.$
\begin{itemize}
 \item If $\Lambda$ is a sampling sequence
 \[
 \liminf_{k\to\infty} \frac{\# ( \Lambda_k \cap \B(x,r) )}{\dim \mathcal P_k}\ge 
 \mu_{eq}(\B(x,r)).
\]
\item If $\Lambda$ is interpolating
 \[
 \limsup_{k\to\infty} \frac{\# ( \Lambda_k \cap \B(x,r) )}{\dim \mathcal P_k}\le 
 \mu_{eq}(\B(x,r)).
\]
\end{itemize}
\end{thm}

\begin{remark}
One can construct interpolation or sampling
sequences with density arbitrary close to the critical density with 
sequences of points $\{\Lambda_k\}$ such that the corresponding Lagrange interpolating polynomials are uniformly bounded.
In particular de above inequalities are sharp, for a similar construction on the sphere see \cite{MOC10}.
\end{remark}

\subsection{Orthonormal basis of reproducing kernels}

Sampling and interpolation are somehow dual concepts. Sequences which are both sampling and interpolating (i.e. complete interpolating sequences)
are optimal in some sense because they are at the same time minimal sampling sequences and maximal interpolating sequences. 
They will satisfy the equality in Theorem \ref{teo:density_ball}.
In general domains, to prove or disprove the existence of such sequences is a difficult problem \cite{OUbook}. 

If 
$\Lambda=\{ \Lambda_k \}$ is a complete interpolating sequence
the corresponding reproducing kernels $\{ \kappa_{k,\lambda} \}$ is a Riesz basis in the space of polynomials (uniformly in the degree).
An obvious example of complete interpolating sequences would be sequences
providing an orthonormal basis of reproducing kernels. In dimension 1, with the weight $(1-x^2)^{a-1/2},$
a basis of Gegenbauer polynomials $\{ G^{(a)}_j \}_{j=0,\dots , k}$ is orthogonal and the reproducing kernel in $\mathcal P_k$ evaluated at the zeros of the polynomial $G^{(a)}_{k+1}$
gives an orthogonal sequence. In our last result we prove that
for greater dimensions there are no  orthogonal basis of $\mathcal P_k$ of reproducing kernels
with the measure $d\mu_a(x)=(1-|x|^2)^{a-1/2}dV(x).$
% Sequences which are both sampling and interpolating, the so called complete interpolating sequences, satisfy 
%  \[
%  \lim_{k\to\infty} \frac{\# ( \Lambda_k \cap \B(x,r) )}{\dim \mathcal P_k}=
%  \mu_{eq}(\B(x,r)),
% \]
% for any   $\B(x,r) \subset \B.$ 

% An easiest example of complete interpolating sequence would be a sequence such that the reproducing kernels form an orthonornal basis of $\mathcal P_k.$ Our next results show that such
% sequences do not exist.

%==============================================================================

% In this section we show that there are no orthonormal basis of reproducing kernels. 

Our first goal is to show that sampling sequences are dense enough, Theorem \ref{noempty}. Recall that in the bulk (i.e. at a fixed positive distance from the boundary) the Euclidean 
metric and the metric $\rho$ are equivalent. In our first result we prove that the right hand side of (\ref{def:sampling}) and the separation imply that 
there are points of the sequence in any ball (of the bulk) of big enough radius.

%==============================================================================

\begin{prop}				\label{prop_bound}
Let $d\mu_a (x)=(1-|x|^2)^{a-\frac{1}{2}}dV(x)$ for $a\ge 0$ the weight in the unit ball $\B\subset \R^n.$
    Let $\Lambda_k\subset \B$ be a finite subset and $C,\epsilon>0$ be constants such that
\begin{equation}
\int_\B |P(x)|^2 d\mu_a(x) \le C \sum_{\lambda\in\Lambda_k} \frac{|P(\lambda)|^2}{K_{k}(\mu_a;\lambda,\lambda)}, 
\end{equation}
  for all $P\in \mathcal{P}_k$
and 
$$\inf_{\substack{\lambda,\lambda'\in \Lambda_k \\ \lambda\neq \lambda' } }\rho(\lambda,\lambda')\ge \frac{\epsilon}{k}.$$ 
Let $|x_0|=C_0<\frac{1}{4},$ $\epsilon<M$ and $k\ge 1$ be such that
$\Lambda_k \cap \mathbb{B}(x_0,M/k)=\emptyset.$ Then $M<A$ for a certain constant $A$ depending only on $C,\epsilon,n$ and $a.$
\end{prop}

%==============================================================================

\proof
  By the construction of function $L_\ell^a (x,y),$ it is clear that for any $\ell\ge 0$
$$K_\ell (\mu_a; x,x)\le \int_{\B }L_\ell^a (x,y)^2 d\mu_a (y)\le   K_{2\ell}(\mu_a; x,x).$$
Let $P(x)=L^a_{[k/2]}(x,x_0)\in \mathcal{P}_k.$ From the property above, the hypothesis and Proposition \ref{prop:Christoffel} we get
\begin{equation}
k^n \sim K_{[k/2]} (\mu_a; x_0,x_0)\le   \int_{\B } P(y)^2 d\mu_a (y)\lesssim 
 \sum_{|\lambda-x_0|>M/k} \frac{|P(\lambda)|^2}{K_k(\mu_a ;\lambda,\lambda)}.
\end{equation}
From \cite[Lemma 11.3.6.]{DX13}, given $x\in \B$ and $0<r<\pi$ 
\begin{equation}
\mu_a (B(x,r))\sim r^n (\sqrt{1-|x|^2}+r)^{2a},
\end{equation}
  and therefore 
\begin{equation}
\mu_a(B(x,r))\sim
\begin{cases}
      r^{n+2a}& \mbox{if}\;\; 1-|x|^2<r^2, \\
      r^n (1-|x|^2)^a & \mbox{otherwise},
\end{cases}
\end{equation}
and
\begin{equation}\label{measure_ball}
\mu_a (B(x,r))\gtrsim
\begin{cases}
      r^{n+2 a}& \mbox{if}\;\; |x|>\frac{1}{2}, \\
      r^n  & \mbox{otherwise}.
\end{cases}
\end{equation}

From (4) in Theorem \ref{theorem_PX}, the separation of the sequence, and the estimate \eqref{measure_ball} we get
\begin{equation}
\begin{split}
 0 & <c  \le   \sum_{|\lambda-x_0|>M/k}  \frac{1}{(1+[k/2]\rho(x_0,\lambda))^{2\gamma}}
 \\
 &
 =
\sum_{|\lambda-x_0|>M/k} \frac{1}{ \mu_a (B(\lambda,\epsilon/2k))}  \int_{B(\lambda,\epsilon/2k)} \frac{d \mu_a (x)}{(1+[k/2]\rho(x_0,\lambda))^{2\gamma}} 
\\
&
\lesssim 
\left[ \sum_{\frac{M}{k}<|\lambda-x_0|<\frac{1}{2}}+ \sum_{\frac{1}{2}<|\lambda-x_0|} \right] \frac{1}{ \mu_a (B(\lambda,\epsilon/2k))}  \int_{B(\lambda,\epsilon/2k)} 
\frac{d \mu_a (x)}{(1+ 2 k \rho(x_0,x))^{2 \gamma}} 
\\
&
\lesssim 
\left(\frac{k}{\epsilon}\right)^n \int_{\frac{M}{k}}^{\frac{3}{4}} \frac{r^{n-1}}{(k r)^{2\gamma}}dr + 
\frac{k^{2a +n-2 \gamma}}{\epsilon^{2 a +n}} \mu_a (B(0,1/2)^c).%\lesssim \frac{M^{d-\beta}}{\epsilon^d}+\frac{k^{2a +n-2\beta }}{\epsilon^{2 a+n}}.
\end{split}
\end{equation}
Now, for $\gamma=n+a$ we get
$$0<c\le \frac{1}{k^{n+2a}}\left[ -\frac{1}{r^{n+2a}}  \right]_{r=\frac{M}{k}}^{\frac{3}{4}}+ \frac{1}{k^n} ,$$
and then a uniform (i.e. independent of $k$) upper bound for $M <A=A(C,\epsilon,n,a).$
\qed

%==============================================================================

%==============================================================================

\begin{prop}				\label{noempty}
    Let $\Lambda=\{\Lambda_k\}$ be a separated sampling sequence for $\B\subset \R^n.$ Then there exist
$M_0,k_0>0$ such that for any $M>M_0$ and all $k\ge k_0$
$$\# \left( \Lambda_k \cap \mathbb{B}(0,M/k)\right)\sim M^n.$$
\end{prop}

%==============================================================================

\proof
    Let $\epsilon>0$ be the constant from the separation, i.e. 
$$\inf_{\substack{\lambda,\lambda'\in \Lambda_k \\ \lambda\neq \lambda' } }\rho(\lambda,\lambda')\ge \frac{\epsilon}{k}.$$ 
Assume that $M/k\le \frac{1}{2}.$ For $\lambda\in \Lambda_k \cap \mathbb{B}(0,M/k)$ we have 
$V(\B(\lambda,\frac{\epsilon}{k}))\sim (\frac{\epsilon}{k})^n$ and therefore
\begin{equation}
\# \left( \Lambda_k \cap \mathbb{B}(0,M/k)\right) \left(\frac{\epsilon}{k}\right)^n \lesssim \left(\frac{M}{k}\right)^n.
\end{equation}
For the other inequality, take the constant $A$ (assume $A>\epsilon$) given in Proposition \ref{prop_bound} depending on the sampling and the separation constants of $\Lambda$ and $n.$ For $M>A$ and $k>0$ such that
$\mathbb{B}(0,\frac{M}{k})\subset \mathbb{B}(0,\frac{1}{4})$ one can find $N$ disjoint balls $\mathbb{B}(x_j,\frac{A}{k})$ for $j=1,\dots N$ included in $\mathbb{B}(0,M/k)$ and
such that 
$$N V(\mathbb{B}(0,\frac{A}{k}))>\frac{1}{2} V(\mathbb{B}(0,\frac{M}{k})).$$
Observe that each ball $\mathbb{B}(x_j,\frac{A}{k})$ contains by Proposition \ref{prop_bound} at least one point from $\Lambda_k$ and therefore 
$$\# \left( \Lambda_k \cap \mathbb{B}(0,M/k)\right)\ge N \gtrsim \left(\frac{M}{A}\right)^n.$$

% we apply Proposition \ref{prop_bound} to get that
% $$
% \# \left( \Lambda_k \cap \mathbb{B}(0,A/k)\right)\ge 1
% $$
\qed

%==============================================================================

We will use the following result from \cite{Fug01}.

\begin{thm}                                                                                         \label{Fuglede-theorem}
    Let $\mathbb{B} \subset \mathbb{R}^{n},$ $n>1,$ be the unit ball. There do not exist infinite subsets
    $\Lambda\subset \mathbb{R}^{n}$ such that the exponentials $e^{i\langle x, \lambda \rangle},$ $\lambda \in \Lambda,$ are pairwise
    orthogonal in $L^{2}(\B ).$ Or, equivalently, there do not exist infinite subsets
    $\Lambda\subset \R^{n}$ such that $|\lambda-\lambda'|$ is a zero of $J_{n/2},$ the Bessel function
    of order $n/2,$ for all distinct $\lambda,\lambda'\in \Lambda.$      
\end{thm}

Following ideas from \cite{Fug74} we can prove now our main result about orthogonal basis. A similar argument can be used on the sphere to study tight spherical designs.

\begin{thm}				\label{noONbasis}
  Let $\B\subset \R^n$ be the unit ball and $n>1.$ There is no sequence of finite sets $\Lambda=\{ \Lambda_k \}\subset \B$ such that the reproducing kernels 
$\{K_k(\mu: x,\lambda) \}_{\lambda\in \Lambda_k}$ form an orthogonal basis of $\mathcal{P}_k$ with respect to the measure $d\mu_a=(1-|x|^2)^{a-\frac{1}{2}}dV$.
\end{thm}

\proof[Theorem \ref{noONbasis}]

The following result can be easily deduced from \cite[Theorem 1.7]{KL13}:

Given $\{u_k\}_k,\{v_k\}_k$ convergent sequences in $\R^n$ and
$u_k\to u, v_k\to v,$ when $k\to \infty.$ Then
$$\lim_{k\to \infty} \frac{K_k(\mu;\frac{u_k}{k},\frac{v_k}{k})}{K_k(\mu;0,0)}=\frac{J^*_{n/2}(|u-v|)}{J^*_{n/2}(0)}.$$

Let $\Lambda_k$ be such that $\{\kappa_{\lambda}\}_{\lambda \in \Lambda_k}$ is an orthonormal basis of 
$\mathcal{P}_k$ with respect to the measure $d\mu_a=(1-|x|^2)^{a-\frac{1}{2}}dV$. Then
$$K_k(\mu;\lambda_{(k)},\lambda_{(k)}')=0,$$
for $\lambda_{(k)}\neq \lambda_{(k)}'\in \Lambda_k.$

 We know that $\Lambda_k$ is uniformly separated for some $\epsilon>0$
$$\rho(\lambda_{(k)},\lambda_{(k)}')\ge \frac{\epsilon}{k}.$$ 
Then the sets $X_k=k(\Lambda_k\cap \mathbb{B}(0,1/2))\subset \R^n$ are  uniformly separated
$$|\lambda- \lambda'|\gtrsim \epsilon,\;\;\lambda\neq \lambda'\in X,$$
 and
$X_k$ converges weakly to some uniformly separated set $X\subset \R^n.$ The limit is not empty because
by Proposition \ref{noempty} for any $M>0,$ 
$$\# \left( \Lambda_k \cap \mathbb{B}(0,M/k)\right)\sim M^d.$$
Observe that this last result would be a direct consequence of the necessary density condition for complete interpolating sets if we could
take balls of radius $r/n$ for a fixed $r>0$ in the condition. Finally, we obtain an infinite set $X$ such that for $\lambda\neq \lambda'\in X$
$$J^*_{n/2}(|\lambda- \lambda'|)=0,$$
in contradiction with Theorem \ref{Fuglede-theorem}.
\qed

\begin{remark}
Note that the fact that the interpolating sequence $\{\Lambda_k\}$ is complete was 
used only to guarantee that $\# \left( \Lambda_k \cap \mathbb{B}(0,M/k)\right)\sim M^d$. So, the above result could be extended to 
sequences $\{\Lambda_k\}$ such that $\{\kappa_{k,\lambda}\}_{\lambda\in \Lambda_k}$ is orthonormal (but not necessarily a basis for $\P_k$) if $\Lambda_k \cap \mathbb{B}(0,M/k)$ contains enough points.
\end{remark}

% \subsection{Open problems}
% 
% \begin{itemize}
%  \item Are there orthogonal basis of reproducing kernels in convex sets of $\R^n$ when $n>1$? And Riesz basis?
%  \item Find asymptotic estimates of the reproducing kernel (or the Christoffel function) for the space of polynomials restricted to $\Omega$ a smoothly bounded convex domain in $\R^n$ with the equilibrium measure. It seems natural to supose that
% $$K_k(\mu_{eq},x,x) = \|K_k(\mu_{eq} ,x,\cdot)\|_{L^2(\mu_{eq})}^2 \simeq k^n.$$
% In fact, it is known that this is the behaviour for any measure similar to the Lebesgue measure inside the support, \cite[Lemma 5.2]{KL13}.
% 
% \end{itemize}

%==============================================================================

\end{document}